\documentclass[12pt,oneside,reqno]{amsart}

\usepackage{hyperref}

\usepackage[headsep=30pt,footskip=30pt,twoside=false,bottom=2.5cm]{geometry}

\usepackage[all]{xy}

\xymatrixcolsep{2cm}

\numberwithin{equation}{section}

\allowdisplaybreaks[1]

\usepackage{etoolbox}

\makeatletter

\patchcmd{\thesubsection}{\arabic}{\arabic}{}{}

\patchcmd{\@seccntformat}{\@secnumfont}{%
  \@secnumfont\expandafter\protect\csname format#1\endcsname}{}{}

\patchcmd{\@startsection}{\@afterindenttrue}{\@afterindentfalse}{}{}

\patchcmd{\subsection}{-.5em}{.3\linespacing}{}{}

\makeatother

\theoremstyle{plain}

\newtheorem{theorem}{Theorem}[section]

\newtheorem{lemma}[theorem]{Lemma}

\theoremstyle{remark}

\newtheorem{remark}[theorem]{Remark}

\newcommand{\Ker}[1]{\ensuremath{\mathrm{Ker} (#1)}}

\newcommand{\At}[1]{\ensuremath{\mathrm{At} (#1)}}

\newcommand{\ENd}[1]{\ensuremath{\mathrm{End}  (#1)}}

\newcommand{\cat}[1]{\ensuremath{\mathcal{#1}}}

\newcommand{\at}[2][]{\ensuremath{\mathrm{at}_{#1} (#2)}}

\newcommand{\id}[1]{\ensuremath{\mathbf{1}_{#1}}}

\newcommand{\rk}[2][]{\ensuremath{\mathrm{rk}_{#1}(#2)}}

\newcommand{\Z}{\ensuremath{\mathbb{Z}}}

\newcommand{\tr}[1]{\ensuremath{\mathrm{tr}(#1)}}

\newcommand{\struct}[1]{\ensuremath{\mathcal{O}_{#1}}}

\newcommand{\Diff}[4][]{\ensuremath{\mathrm{Diff}^{#1}_{#2}(#3,#4)}}

\newcommand{\coh}[3]{\ensuremath{\mathrm{H}^{#1}(#2,#3)}}

\renewcommand{\bar}[1]{\ensuremath{\overline{#1}}}

\begin{document}

\title[A criterion for the logarithmic connections over a  perfect field]{A criterion for the existence of logarithmic connections on curves over a perfect field}

\author[S. Manikandan]{S. Manikandan}

\address{Indian Institute of Science Education and Research, Tirupati \\
C/o Sree Rama Engineering College (Transit Campus),
Karakambadi Road, Mangalam (P.O.) Tirupati 517507.
Andhra Pradesh, India}

\email{manimaths87@gmail.com}

\author[A. Singh]{Anoop Singh}

\address{School of Mathematics \\ Tata Institute of Fundamental Research \\
Homi Bhabha Road \\ Mumbai 400~005 \\ India }

  \email{anoops@math.tifr.res.in}



\begin{abstract}
Let $k$ be a perfect field, and $X$ an irreducible 
smooth projective curve over $k$. We give a criterion for a vector bundle over $X$ to admit
a logarithmic connection singular over a finite subset 
of $X$ with given residues, where residues are assumed to be rigid.
\end{abstract}

\maketitle

\section{Introduction and statements of the results}
Let $X$ be a compact Riemann surface. A famous theorem due to Atiyah \cite{A} and Weil \cite{W}, which is known together as the Atiyah-Weil criterion, says that a holomorphic vector bundle over a $X$ admits a
holomorphic connection if and only if the degree of each indecomposable component of the holomorphic
vector bundle is zero (see \cite{BN} for an exposition of the Atiyah-Weil criterion). In \cite{BS} and 
\cite{B06}, the Atiyah-Weil criterion has been 
generalised for the smooth projective curve over infinite perfect field,
and perfect field, respectively.
 In \cite{B},
a criterion for the existence of a logarithmic connection with prescribed residues has been 
established, and hence generalising the Atiyah-Weil criterion in logarithmic set up. More precisely,
let 
$S = \{x_1, \ldots, x_m\}$ be a subset of $X$ such that 
$x_i \neq x_j$ for all $i \neq j$, and let $E$ be a holomorphic 
vector bundle over $X$. 
 Fix  a {\bf rigid} 
endomorphism $A(x) \in \ENd{E(x)}$ for every $x \in S$, where $E(x)$ denote the fibre of $E$ over $x \in S$.  Then, we have the following.
\begin{theorem}\cite[Theorem 1.3]{B}
\label{thm:1}
The vector bundle $E$ admits a logarithmic connection
singular over $S$ with residues $A(x)$ at every $x \in S$ if and only if for every direct summand $F \subset E$,
\begin{equation}
\label{eq:1}
\deg F + \sum_{x \in S} \tr{A(x)\vert_{F(x)}} = 0,
\end{equation}
where $F(x)$ denote the fibre of $F$ over $x \in S$.
\end{theorem}

The proof in \cite[Theorem 1.3]{B} will work for vector 
bundles on a smooth projective curve defined over 
an algebraically closed field of characteristic zero.

Motivated by the above discussion, we have problems related to 
existence of logarithmic connections when the curve 
is over an algebraically closed field of characteristic 
$p > 0$, a prime number.
Also, what will be the suitable criterion  when the base 
field $k$ fails to be an 
algebraically closed field. 
 
In case $k$ is an  algebraically closed field of characteristic 
$p > 0$, and $X$ is an irreducible smooth projective 
curve over $k$, we prove the following (see section \ref{alg. closed field} 
Theorem \ref{thm:2}).
\begin{theorem}
\label{thm:0.2}
Let $E$ be an algebraic vector bundle on $X$ defined over algebraically closed field $k$ of characteristic 
$p > 0$. Then $E$ admits a logarithmic connection
singular over $S$ with residue $A(x)$ for every 
$x \in S$ if and only if every indecomposable component 
$F$ of $E$ satisfies the following condition
\begin{equation}
\label{eq:0.15}
\deg{F} + \sum_{x \in S} \tr{A(x)\vert_{F(x)}} \equiv 0 ~(\mbox{mod}~p),
\end{equation}
that is, the number $\deg{F} + \sum_{x \in S} \tr{A(x)\vert_{F(x)}} \in k$ is a multiple of $p$.
\end{theorem}

By the abuse of notation, we denote the image of $\deg{E} \in \Z$ under the 
morphism 
$$\Z \longrightarrow \Z/{p\Z} \hookrightarrow k$$
by $\deg{E}$ itself, and this is used throughout the paper.

We also show the following result (see section 
 \ref{perfect field},  Theorem \ref{thm:3}), and 
 this will generalise  \cite[Theorem 1.1]{B} in the logarithmic framework.
\begin{theorem}
\label{thm:0.3}
Let $k$ be a perfect field of characteristic $p$. Let $E$ be a vector bundle 
on an irreducible smooth projective curve $X$
over $k$. Then, we have 
\begin{enumerate}
\item \label{0.a}
Assume that $p >0$, and suppose that 
rank of each indecomposable components of $E$ is not 
divisible by $p$. Then $E$ admits a logarithmic connection singular over $S$ with residue $A(x)$ for 
every $x \in S$ if and only if for every indecomposable 
component $F$ of $E$ satisfies
\begin{equation}
\label{eq:0.23}
\deg{F} + \sum_{x \in S} \tr{A(x)\vert_{F(x)}} \equiv 0 ~(\mbox{mod}~p),
\end{equation}
\item \label{0.b} If for every indecomposable component $F$ of $E$
satisfies 
\begin{equation}
\label{eq:0.24}
\deg{F} + \sum_{x \in S} \tr{A(x)\vert_{F(x)}} = 0 ,
\end{equation}
then $E$ admits a logarithmic connection singular 
over $S$ with residue $A(x)$ for every $x \in S$.
\end{enumerate}
\end{theorem}

\section{Logarithmic connection and residues over a field}
Let $k$ be a field. Let $X$ be an irreducible smooth
projective curve over $k$. Let 
$$S := \{x_1, \ldots, x_m\}$$ be a finite subset of closed points of $X$ such that $x_i \neq x_j$ for $i \neq j$, and  
let $$Z : = x_1 + \ldots + x_m$$ denote the reduce effective 
divisor associated with $S$. Let $\Omega^1_X$ denote the 
cotangent bundle of $X$. Let $E$ be a vector bundle 
over $X$. A logarithmic connection  on $E$ singular over $S$ is a $k$-linear map 
\begin{equation}
\label{eq:3}
D : E \to  E \otimes
\Omega^1_X \otimes \struct{X}(Z)  
\end{equation}
which satisfies the Leibniz identity
\begin{equation}
\label{eq:4}
D(f s)= f D(s) + df \otimes s,
\end{equation}
where $f$ is a local section of \struct{X} and $s$ is a 
local section of $E$.

We shall give an equivalent definition of a logarithmic 
connection in terms of splitting of logarithmic Atiyah exact sequence. 

Let $\Diff[1]{X}{E}{E}$ be the vector bundle over $X$
whose sections over any open subset $U \subset X$ are the  differential operators on $E\vert_U$ of order at most one. Let 
\begin{equation*}
\sigma_1 : \Diff[1]{X}{E}{E} \longrightarrow TX \otimes
\ENd{E}
\end{equation*}
be the symbol operator, where $TX$ be the tangent bundle
of $X$. The symbol operator $\sigma_1$ is surjective.
Consider the {\bf symbol exact sequence} 
\begin{equation}
\label{eq:5}
0 \to \ENd{E} \xrightarrow{\iota} \Diff[1]{X}{E}{E} \xrightarrow{\sigma_1} TX \otimes\ENd{E} \to 0.
\end{equation}
Let $\struct{X}(-Z) \subset \struct{X}$ be the line 
bundle associated to the divisor $-Z$.
Then $$\struct{X}(-Z) \otimes TX \otimes \ENd{E} \subset TX \otimes \ENd{E}.$$
Define a vector bundle on $X$ as follows.
$$\At{E}(- \log Z ) := \sigma_1^{-1}(\id{E} \otimes TX \otimes \struct{X}(-Z)).$$
From the symbol exact sequence \eqref{eq:5}, we get 
a short exact sequence 
\begin{equation}
\label{eq:6}
0 \to \ENd{E} \xrightarrow{\iota} \At{E}(- \log Z ) \xrightarrow{\widetilde{\sigma_1}} TX(- \log Z) \to 0,
\end{equation}
called the {\bf logarithmic Atiyah exact sequence},
where $\widetilde{\sigma_1}$ is the restriction of $\sigma_1$ and $TX (- \log Z) := TX \otimes \struct{X}(-Z).$

Now, $E$ admits a logarithmic connection singular over 
$S$ if and only if the logarithmic Atiyah exact sequence 
\eqref{eq:6} splits algebraically, that is, there exists
an $\struct{X}$-linear homomorphism 
\begin{equation}
\label{eq:7}
\alpha : TX(- \log Z) \to \At{E}(- \log Z)
\end{equation}
such that 
$$\widetilde{\sigma_1} \circ \alpha = \id{TX(- \log Z)}.$$

Next we will define residue of a logarithmic connection
in $E$ at $x \in S$. Fix $x \in S$, and let $U$ be an open subset of $X$ such that $U \cap S = \{x\}$.
Let $s$ be an algebraic  section of $\At{E}(- \log Z )$
over $U$. Then $\widetilde{\sigma_1}(s)$ will be a 
section of $TX(- \log Z)$ over $U$. Now, evaluating 
$\widetilde{\sigma_1}(s)$ at $x$, we get 
$$\widetilde{\sigma_1}(s)(x) = \widetilde{\sigma_1}_x(s(x)) = 0.$$ 
Since $\Ker {\widetilde{\sigma_1}} = \ENd{E}$, we get 
a $k$-linear map
\begin{equation}
\label{eq:8}
j_x : \At{E}(- \log Z )(x) \to \ENd{E}(x) = \ENd{E(x)}.
\end{equation}
Note that 
for any $x \in S$, the fiber $\Omega^1_X 
\otimes \struct{X}(Z)(x)$ is canonically identified 
with $k$ by sending a rational $1$-form to its residue at 
$x$.  

Given a logarithmic connection $D$ on $E$ singular over
$S$, we get a unique algebraic splitting $\alpha$ of
\eqref{eq:6} corresponding to $D$, and thus
we define the residue of $D$ at $x \in S$
by
\begin{equation}
\label{eq:10}
Res(D,x) := j_x(\alpha_x(1)) \in \ENd{E(x)},
\end{equation}
where $1 \in k$ and the fiber $TX(- \log Z)(x)$ is canonically identified 
with $k$ as described above.

Next, we describe logarithmic connections with 
prescribed residues.  For that first notice the following.
\begin{lemma}
\label{lemm:1}
For every $x \in S$,  the fibre $\At{E}(- \log Z )(x)$
has the canonical decomposition
\begin{equation}
\label{eq:11}
\At{E}(- \log Z )(x) = \ENd{E(x)} \oplus k
\end{equation}
\end{lemma}
\begin{proof}
See \cite[Lemma 2.1]{B}.
\end{proof}

Fix $A(x) \in \ENd{E(x)}$ for every $x \in S$. 
Consider the one dimensional vector space over $k$
generated by the vector $(A(x), 1)$ in 
$\ENd{E(x)} \oplus k$, that is,
\begin{equation}
\label{eq:12}
l_x := k.(A(x), 1) \subset \ENd{E(x)} \oplus k = 
\At{E}(- \log Z )(x).
\end{equation}

Let $\cat{A}(E) \to X$ be the vector bundle
that fits in the following short exact sequence
\begin{equation}
\label{eq:13}
0 \to \cat{A}(E) \to \At{E}(- \log Z) \to \bigoplus_{x\in S} \frac{\At{E}(- \log Z)(x)}{l_x} \to 0.
\end{equation}
where $l_x$ is  constructed above.

From logarithmic Atiyah exact sequence \eqref{eq:6} we have the short exact sequence
\begin{equation}
\label{eq:14}
0 \to \ENd{E} \otimes \struct{X}(-Z) \xrightarrow{\iota} \cat{A}(E) \xrightarrow{\widehat{\sigma_1}} TX(- \log Z) \to 0,
\end{equation}
where $\widehat{\sigma_1}$ is the restriction of 
$\widetilde{\sigma_1}$.

\begin{lemma}
\label{lemm:2}
A logarithmic connection in $E$ with given residues
$A(x) \in \ENd{E(x)}$, for every $x \in S$ is an algebraic splitting of 
the short exact sequence \eqref{eq:14}, that is,
there exists a morphism
$$h : TX (- \log Z) \to \cat{A}(E)$$ such that 
$\widehat{\sigma_1} \circ h = \id{TX (- \log Z)} $.
\end{lemma}

We recall the notion of rigid endomorphism.
Let $x \in X$. An endomorphism $\beta \in \ENd{E(x)}$
is said to be a rigid 
if $$\beta \circ \phi(x) = \phi(x) \circ \beta$$
for all $\phi \in \coh{0}{X}{\ENd{E}}$.
Now onwards we shall assume that the endomorphism
$A(x)$ is rigid for every $x \in S$.

Recall that a vector bundle $E$ over $X$ is said to be 
{\it decomposable} if there are vector bundles $F$ and 
$F'$ such that $\rk F >0$, $\rk F' > 0$ and
$$E \cong F \oplus F'.$$ 
A vector bundle is called {\bf indecomposable} if it
is not decomposable.

From \cite[p.315, Theorem 2]{A1},  any vector bundle $E$ over $X$ is isomorphic 
to a unique, up to reordering, direct sum of indecomposable vector bundles, and we call it 
Krull-Remak-Schmidt decomposition of $E$.
Let $$E = \bigoplus_{i = 1}^n E^i$$ be the 
Krull-Remak-Schmidt decomposition of $E$. Since 
$A(x) \in \ENd{E(x)} $ is rigid for every $x \in S$,
$$A(x)(E^i{(x)}) \subset E^i{(x)}.$$ The restriction 
of $A(x)$ on $E^i(x)$ is denoted by $A^i(x)$ for every 
$x \in S$ and for every $i = 1, \ldots, n$.
We use the notations as above for the following.
\begin{lemma}
\label{lemm:3}
$E$ admits a logarithmic connection singular over $S$
with residues $A(x)$ at every $x \in S$ if
and only if each indecomposable component $E^i$
admits a logarithmic connection $D^i$ with 
residue $A^i(x)$ at every $x \in S$.
\end{lemma}
\begin{proof}
Let $\iota : E^i \to E$ be the inclusion map, and $q^i : E \to E^i$  the quotient map.
Let $$D : E \to E \otimes \Omega^1_X(\log Z)$$ be a logarithmic connection singular over $S$ with residue
$A(x)$ at every $x \in S$. Consider the composition
$$(q^i \otimes \id{\Omega^1_X(\log Z)}) \circ D \circ \iota : E^i \to E^i \otimes \Omega^1_X(\log Z),$$
which satisfies the Leibniz rule and singular over $S$
with residue $A^i(x)$ at every $x \in S$.
Conversely, given logarithmic connection $D^i$ on $E^i$
for every $1 \leq i \leq n$, singular over $S$
with residue $A^i(x)$ at every $x \in S$.
Then $\bigoplus_{i = 1}^n D^i$ gives a logarithmic connection on $E$ singular over $S$ with residue
$A(x)$ at every $x \in S$.
\end{proof}

\section{Criterion over an algebraically closed field of characteristic $p > 0$}
\label{alg. closed field}
In this section, we assume that $k$ is an algebraically closed filed of characteristic $p > 0$, and $X$  an irreducible 
smooth projective curve over $k$. 
\begin{theorem}
\label{thm:2}
Let $E$ be an algebraic vector bundle on $X$ defined over the algebraically closed field $k$ of characteristic 
$p > 0$. Then $E$ admits a logarithmic connection
singular over $S$ with residue $A(x)$ for every 
$x \in S$ if and only if every indecomposable component 
$F$ of $E$ satisfies the following condition
\begin{equation}
\label{eq:15}
\deg{F} + \sum_{x \in S} \tr{A(x)\vert_{F(x)}} \equiv 0 ~(\mbox{mod}~p),
\end{equation}
that is, the number $\deg{F} + \sum_{x \in S} \tr{A(x)\vert_{F(x)}} \in k$ is a multiple of $p$.
\end{theorem}
\begin{proof}
In view of Lemma \ref{lemm:3}, it is enough to 
prove the theorem for indecomposable vector bundles.
Without loss of generality assume that $E$ is an 
indecomposable vector bundle. Consider the 
short exact sequence \eqref{eq:14} associated with
$E$ and its extension class (called logarithmic Atiyah class) 
\begin{equation}
\label{eq:16}
\phi^A_E \in \coh{1}{X}{\ENd{E} \otimes \Omega^1_X}.
\end{equation}
By Serre duality, the logarithmic Atiyah class $\phi^A_E$ corresponds to an element 
\begin{equation}
\label{eq:17}
\widetilde{\phi^A_E} \in \coh{0}{X}{\ENd{E}}^*.
\end{equation}
Now, we shall construct a $k$-linear morphism
\begin{equation}
\label{eq:18}
\delta_x : \ENd{E(x)} \to \coh{1}{X}{\ENd{E} \otimes \Omega^1_X}
\end{equation}
of vector spaces for every $x \in S$. In fact, this morphism $\delta_x$ can be constructed for any $x \in X$.
Consider the short exact sequence

\begin{equation}
\label{eq:19}
0 \to \ENd{E} \otimes \Omega^1_X \rightarrow \struct{X}(x) \otimes \ENd{E} \otimes \Omega^1_X  \xrightarrow{Res_x} \ENd{E(x)} \to 0,
\end{equation}
 where the last map is given by sending logarithmic $1$-form with values in $\ENd{E}$ to its residue at 
$x$ with values in $\ENd{E(x)}$.
The morphism $\delta_x$ in \eqref{eq:18} is the coboundary 
operator in the long exact sequence of cohomology 
groups induced from the short exact sequence 
$\eqref{eq:19}$.
From Serre duality, and from $\delta_x$, we get an induced morphism
\begin{equation}
\label{eq:20}
\widetilde{\delta_x} : \ENd{E(x)} \to \coh{0}{X}{\ENd{E}}^*,
\end{equation}
such that 
$$\widetilde{\delta_x}(\beta)(\gamma) = \tr{\beta \circ \gamma(x)}.$$
Now, consider the Atiyah exact sequence (see \cite{A}
and \cite[Proposition 4.2]{BS1} for the proof) 
\begin{equation}
\label{eq:21}
0 \to \ENd{E} \xrightarrow{\iota} \At{E} \xrightarrow{\sigma_1} TX \to 0,
\end{equation}
and let $\at{E} \in \coh{1}{X}{\ENd{E} \otimes \Omega^1_X}$ denote the extension class,
known as Atiyah class. Again using Serre duality,
we have $\widetilde{\at{E}} \in \coh{0}{X}{\ENd{E}}^*$.
We claim that
\begin{equation}
\label{eq:22}
\widetilde{\phi^A_E} = \widetilde{\at{E}} + \sum_{x \in S} \widetilde{\delta_x}(A(x))
\end{equation}
Since $k$ is an algebraically closed field and $E$
is indecomposable,  any element $\theta \in \coh{0}{X}{\ENd{E}}$ is of the form 
$$\theta = \nu \id{E} + N,$$
where $\nu \in k$ and $N$ is a nilpotent endomorphism 
of $E$.
To prove \eqref{eq:22}, it is enough to verify the formula \eqref{eq:22}
for $\id{E}$ and $N$ separately. From
\cite[Proposition 18(ii)]{A},
it is know that $\widetilde{\at{E}}(N) = 0$.
Next, since $N \in \ENd{E}$ is nilpotent, $N_x \in \ENd{E(x)}$ is nilpotent. Since $A(x)$ is a rigid
endomorphism, $N(x)$ commutes with $N(x)$ and 
hence $A(x) \circ N(x)$ is nilpotent.
Thus, $$\widetilde{\delta_x}(A(x))(N) = \tr{A(x) \circ N(x)} = 0.$$
Let $E_{\bullet}$ be a flag 
$$0 = E_0 \subset E_1 \subset \ldots \subset E_l = E$$
of subbundles of $E$ such that $$A(x) (E_i(x)) \subset E_i(x),$$ for every $x \in S$ and for every $1 \leq i \leq l$.
Let $\ENd{E}^0$ be the $\struct{X}$-submodule of
$\ENd{E}$ consisting of endomorphisms which preserves
the flag $E_\bullet$.
Then $\ENd{E}^0$ is a subbundle of $\ENd{E}$.
The inclusion morphism $\iota :\ENd{E}^0 \hookrightarrow \ENd{E}$, induces a morphism 
$$\iota^* : \coh{1}{X}{\ENd{E}^0 \otimes \Omega^1_X} \to \coh{1}{X}{\ENd{E} \otimes \Omega^1_X}.$$ It can be shown that 
the logarithmic Atiyah class $\phi^A_E$ is in the image 
of $\iota^*$.

Moreover, for a nilpotent endomorphism $N$, the subbundles $\Ker{N^j}$ of $E$ for $j \geq 1$, form a flag of 
$E$ such that $$A(x) (\Ker{N^j}(x)) \subset \Ker{N^j}(x), $$ because
$A(x) \circ N^j(x) = N^j(x) \circ A(x)$ for every
$j \geq 1$. We take $\ENd{E}^0$ to be associated with the flag  $\Ker{N^j}$, $j \geq 1$. Then, from above observation $\widetilde{\phi^A_E}(N) = 0$.
Thus, \eqref{eq:22}  satisfies for $N$.

Consider the following composition of morphisms
\begin{align*}
\coh{0}{X}{\ENd{E}} \otimes \coh{1}{X}{\ENd{E} \otimes \Omega^1_X} \xrightarrow{\cup} \coh{1}{X}{\ENd{E} \otimes \ENd{E} \otimes \Omega^1_X} \\
 \xrightarrow{m^*} \coh{1}{X}{\ENd{E} \otimes \Omega^1_X} \xrightarrow{\mbox{trace}} \coh{1}{X}{\Omega^1_X} = k,
\end{align*}
and denote it by $\Psi$, where the first morphism is the
cup product, second morphism $m^*$ is induced from
$$m : \ENd{E} \otimes \ENd{E} \to \ENd{E} $$
which is composition of endomorphisms of $E$.

Under this composition $\Psi$, the image of 
$\id{E} \otimes \phi^A_E$ coincides with $\deg{E} + \sum_{x\in S} \tr{A(x)} \in k$, that is,
$$\widetilde{\phi^A_E} (\id{E}) = \deg{E} + \sum_{x \in S} \tr{A(x)} .$$

Also, under the same morphism $\Psi$, the image of 
$\at{E}$ coincides with $\deg{E}$(see \cite[Proposition 3.1]{BS}), and hence 
$$\widetilde{\at{E}}(\id{E}) = \deg{E}.$$

As observed above, $\widetilde{\delta_x}(A(x))(\id{E}) = \tr{A(x) \circ \id{E}} = \tr{A(x)}$.
Thus, \eqref{eq:22} verifies for $\id{E}$, and hence completes the proof of the claim.

Thus, $E$ admits a logarithmic connection singular over
$S$ with given residue $A(x)$ for every $x \in S$
if and only if the obstruction class $\phi^A_E$ vanishes
which is equivalent to $\widetilde{\phi^A_E}(\id{E}) = 0$, that is, 
$\deg{E} + \sum_{x \in S} \tr{A(x)}$ vanishes in $k$,
which is nothing but
\begin{equation*}
\deg{E} + \sum_{x \in S} \tr{A(x)} \equiv 0 ~(\mbox{mod}~p).
\end{equation*}
This completes the proof of the theorem.
\end{proof}
\begin{remark}\mbox{}
\label{rem:1}
In the proof of the Theorem \ref{thm:2}, we observed 
that $\widetilde{\phi^A_E} (\id{E}) \in k$ corresponds 
to $\deg{E} + \sum_{x \in S} \tr{A(x)} \in k$. In
fact, this is true for every field $k$. Therefore,
if a vector bundle $E$ on an irreducible smooth projective curve $X$ defined over any field $k$ admits 
a logarithmic connection singular over $S$ with
residue $A(x)$ for every $x \in S$, then the number $$\deg{E} + \sum_{x \in S} \tr{A(x)}$$
is a multiple of the characteristic of $k$.
\end{remark}

\section{Criterion over a perfect field}
\label{perfect field}
In this section, we prove an analogous result, when the
base field fails to be algebraically closed. In this section, we will assume that
$k$ is a perfect field of arbitrary characteristic,
and $X$ is an irreducible smooth projective
curve defined over $k$.
\begin{theorem}
\label{thm:3}
Let $k$ be a perfect field of characteristic $p$. Let $E$ be a vector bundle 
on an irreducible smooth projective curve $X$
over $k$. Then, we have 
\begin{enumerate}
\item \label{a}
Assume that $p >0$, and suppose that 
rank of each indecomposable components of $E$ is not 
divisible by $p$. Then $E$ admits a logarithmic connection singular over $S$ with residue $A(x)$ for 
every $x \in S$ if and only if for every indecomposable 
component $F$ of $E$ satisfies
\begin{equation}
\label{eq:23}
\deg{F} + \sum_{x \in S} \tr{A(x)\vert_{F(x)}} \equiv 0 ~(\mbox{mod}~p),
\end{equation}
\item \label{b} If for every indecomposable component $F$ of $E$
satisfies 
\begin{equation}
\label{eq:24}
\deg{F} + \sum_{x \in S} \tr{A(x)\vert_{F(x)}} = 0 ,
\end{equation}
then $E$ admits a logarithmic connection singular 
over $S$ with residue $A(x)$ for every $x \in S$.
\end{enumerate}
\end{theorem} 

In view of the Remark \ref{rem:1}, in the first part \eqref{a} of the above Theorem \ref{thm:3}, one 
direction is obvious, that is, if $E$ admits 
a logarithmic connection singular over $S$ with
residue $A(x)$ for every $x \in S$, then $$\deg{F} + \sum_{x \in S} \tr{A(x)\vert_{F(x)}} \equiv 0 ~(\mbox{mod}~p).$$
The hard part is to prove the converse of it. Also,
notice that the converse of the second part \eqref{b}
of the Theorem \ref{thm:3} is not true.
In order to prove above Theorem \ref{thm:3}, we need the notion of {\bf absolutely indecomposable} vector bundles
over $X$ and some key properties of finite extensions
of perfect fields.

Let $K$ be a finite extension of $k$.
Let $$X_K = X \times_{\mbox{spec(k)}} \mbox{spec(K)} $$
be the  curve obtained from base change and let
\begin{equation}
\label{eq:25}
\pi : X_K \longrightarrow X
\end{equation}
be the natural 
projection. Consider the subset $S = \{x_1, \ldots, x_m\}$  of $X$ as above.
Then $$\widetilde{S} := \pi^{-1}(S) =\{\pi^{-1}(x_1), \ldots, \pi^{-1}(x_m)\}$$ will be a subset consists of distinct closed 
points of $X_K$. For simplicity we set 
$y_j = \pi^{-1}(x_j)$ for every $1 \leq j \leq m$.
 Let $$\widetilde{Z} := y_1 + \cdots + y_m$$ be the reduced
 effective divisor associated with $\widetilde{S}$.
 
 Let $\cat{F}$ be a vector bundle over $X_K$, and let
 \begin{equation}
 \label{eq:26}
 F : = \pi_* \cat{F}
 \end{equation}
  be the direct image of 
 the vector bundle $\cat{F}$ under $\pi$. Then $F$ is a 
 vector bundle over $X$. Let
 $$d := \deg{\pi} = [K:k]$$ denote the 
 degree of the extension $K\vert k$. Then
 \begin{equation}
 \label{eq:27}
 \deg{F} = d \cdot \deg{\cat{F}},
 \end{equation}
and we get a relation between their ranks
\begin{equation}
\label{eq:28}
\rk{F} = d \cdot \rk{\cat{F}}.
\end{equation}
Under above notations we have 
\begin{lemma}
\label{lemm:4}
If the vector bundle $\cat{F}$ over $X_K$ admits a logarithmic connection 
 singular over $\widetilde{S}$, then the vector bundle $F$ $(\mbox{defined}~in~\eqref{eq:26})$
over $X$ admits 
a logarithmic connection singular over $S$.
\end{lemma}
\begin{proof}
Let $$D : \cat{F} \to \cat{F} \otimes \Omega^1_{X_K}
\otimes \struct{X_K}(-\widetilde{Z}) = \cat{F} \otimes 
\Omega^1_{X_K}(\log \widetilde{Z})$$
be a logarithmic connection on $\cat{F}$ singular over
$\widetilde{S}$.

Let $U \subset X$ be an open subset, and 
$s \in \cat{F}(\pi^{-1}(U))$. Then $D(s)$ be a local 
section of $\cat{F} \otimes \Omega^1_{X_K}(\log \widetilde{Z})$ over 
$\pi^{-1}(U)$. This implies that $D(s)$ is a local section of $\pi_*(\cat{F} \otimes 
\Omega^1_{X_K}(\log \widetilde{Z}))$ over $U$.
Since $$\pi^* (\Omega^1_X(\log Z)) = \Omega^1_{X_K}(\log \widetilde{Z}),$$ from the projection formula, we get

\begin{align*}
\pi_*(\cat{F} \otimes \Omega^1_{X_K}(\log\widetilde{Z}))
& = \pi_*(\cat{F} \otimes \pi^* (\Omega^1_X(\log Z)) \\
& = \pi_*(\cat{F}) \otimes \Omega^1_X(\log Z)\\
& = F \otimes \Omega^1_X(\log Z).
\end{align*}
Thus, $D(s)$ gives a local section of $F \otimes \Omega^1_X(\log Z)$ over $U$.

We define an operator $$\nabla : F \to 
F \otimes \Omega^1_X(\log Z)$$ by 
$$\nabla(\hat{s}) := D(s),$$
where $\hat{s} \in F(U)$, a section given by  
$s \in (\pi_*\cat{F})(U) = F(U)$, for every open subset
$U \subset X$. 

For any $f \in \struct{X}(U)$, $d(f \circ \pi) = \pi^* df \in \Omega^1_{X_K}(\log \widetilde{Z})(U)$, where $d : \struct{X} \to \Omega^1_X(\log Z)$ is the 
universal derivation. The operator $\nabla$ satisfies the following property (Leibniz rule).
\begin{align*}
\nabla (f \cdot \hat{s}) & = D ((f \circ \pi) \cdot s) \\
& = (f \circ \pi) D(s) + d(f \circ \pi) \otimes s \\
& = f \nabla(\hat{s}) +  df \otimes \hat{s}.
\end{align*}

Thus, $\nabla$ is a logarithmic connection on $F$
singular over $S$.
 \end{proof}

\begin{remark}\mbox{}
\label{rem:2}
In the above Lemma \ref{lemm:4}, we can choose residue 
$B(y)$ for every $y \in \widetilde{S}$,
for the logarithmic connection $D$ on $\cat{F}$ singular 
over $\widetilde{S}$, such that it maps to the 
given residue $A(x)$ for every $x \in S$, for the
induced logarithmic connection $\nabla$ on $F$  singular over $S$,
under $\pi$.
\end{remark}

Let $Y$ be a geometrically irreducible projective curve defined over a field $l$.
For the following definition of absolutely indecomposable vector bundle, see \cite{AEJ} and \cite[Definition 2.2]{B06}.

 A vector bundle $\cat{E}$ over $Y$ is called {\bf absolutely
indecomposable} if there is an algebraic closure $\bar{l}$ of $l$ such that the corresponding
vector bundle $\cat{E} \otimes_{l} \bar{l}$ over $Y \times_{l} \bar{l}$ is indecomposable.

In fact, the
condition that a vector bundle over $Y$ is absolutely indecomposable does not depend on the choice of the algebraic closure $\bar{l}$ (see the first paragraph  on  \cite[p.n. 86]{B06}).

\begin{proof}[\bf Proof of Theorem \ref{thm:3}]
In view of Lemma \ref{lemm:3}, it is enough to show the 
theorem for indecomposable vector bundles. To prove the 
first part \eqref{a}, let $k$ be a perfect field of 
characteristic $p > 0$, and $E$ be an indecomposable vector bundle over $X$
such that
\begin{equation}
\label{eq:28.5}
\deg{E} + \sum_{x \in S} \tr{A(x)} \equiv 0 ~(\mbox{mod}~p),
\end{equation}
where $A(x)$ is given rigid endomorphism on $E(x)$, for 
every $x \in S$.

Since $k$
is a perfect field, from \cite[Theorem 1.8, 4]{AEJ} there is a finite field extension $K$ of $k$ with the following property.

There is an absolutely indecomposable vector bundle $\\cat{E}$ over $$X_K := X \times_{k} K $$ such that  
\begin{equation}
\label{eq:29.5}
\pi_*\cat{E} \cong  E,
\end{equation}
where $\pi : X_K \to X$ is the
natural projection. The above result is due to A.
Tillmann\cite{T} as mentioned in \cite{AEJ}. 

In view of \eqref{eq:28},
for every $y \in \widetilde{S}$, choose endomorphism $B(y) \in \cat{E}(y) \cong K^{\rk{\cat{E}}}$ such that 
\begin{equation}
\label{eq:29}
d \cdot \tr{B(y)} = \tr{A(x)},
\end{equation}
where $\pi (y) = x$, and  $d$ the degree of extension $K \vert k$.

Note that $\tr{B(y)} \in K$ and $\tr{A(x)} \in k$,
so we consider $K$ as a vector space over $k$, and we take trace of $\tr{B(x)} \in K$ to get \eqref{eq:29}.

From \eqref{eq:28.5}, and \eqref{eq:27}, we have
\begin{equation}
\label{eq:30}
d \cdot \deg{\cat{E}} + \sum_{y \in \widetilde{S}} d \cdot \tr{B(y)} 
\equiv 0 ~(\mbox{mod}~p).
\end{equation}
Since $\rk{E}$ is not divisible by $p$, from 
\eqref{eq:28}, $d$ is coprime to $p$, and hence 
\eqref{eq:30} becomes,
\begin{equation}
\label{eq:31}
\deg{\cat{E}} + \sum_{y \in \widetilde{S}}  \tr{B(y)} 
\equiv 0 ~(\mbox{mod}~p).
\end{equation}

Since $\cat{E}$ is an absolutely indecomposable vector
bundle over $X_K$, by above definition 
$$\cat{E}_{\bar{K}} := \cat{E} \otimes_K{\bar{K}}$$ is 
indecomposable vector bundle 
over $X_{\bar{K}}:= X_K \times_K \bar{K}$, where 
$\bar{K}$ is an algebraic closure of $K$.

Let $\widehat{\pi} : X_{\bar{K}} \longrightarrow X_K$ 
be the 
natural projection. Choose a subset
 $$\widehat{S} = \{\hat{y_1}, \ldots, \hat{y_m}\} 
 \subset X_{\bar{K}}$$ such that 
 $\widehat{\pi}(\hat{y_j}) = y_j$
 for every $j = 1, \ldots, m$. 
 Notice that $\deg{\cat{E}_{\bar{K}}} = \deg{\cat{E}}$.
 From above observation,
for every $\hat{y} \in \widehat{S}$, 
choose endomorphisms
$\widehat{B(\hat{y})}$ on $\cat{E}_{\bar{K}}(\hat{y})$
such that $\tr{\widehat{B(\hat{y})}} = \tr{B(y)}$.
Thus, the equation \eqref{eq:31} becomes
\begin{equation}
\label{eq:32}
\deg{\cat{E}_{\bar{K}}} + \sum_{\hat{y} \in \widehat{S}}  
\tr{\widehat{B(\hat{y})}}
\equiv 0 ~(\mbox{mod}~p),
\end{equation}
 and from Theorem \ref{thm:2}, $\cat{E}_{\bar{K}}$
 admits a logarithmic connection singular over 
 $\widehat{S}$ with residue $\widehat{B(\hat{y})}$ for 
 every $\hat{y} \in \widehat{S}$.
 
Consider the short exact sequence \eqref{eq:14}
for the vector bundle $\cat{E}$ over $X_K$ and  for the 
given endomorphism $B(y) \in \ENd{\cat{E}(y)}$ for 
every $y \in \widetilde{S}$, that is,
\begin{equation}
\label{eq:33}
0 \to \ENd{\cat{E}} \otimes \struct{X_K}(- \widetilde{Z}) \xrightarrow{\iota} \cat{A}(\cat{E}) \xrightarrow{\widehat{\sigma_1}} TX_{K}(- \log \widetilde{Z}) \to 0.
\end{equation}
Let $\phi^B_{\cat{E}} \in \coh{1}{X_K}{\ENd{\cat{E}} \otimes \Omega^1_{X_K}}$ be the extension class of 
\eqref{eq:33}.

Changing the base from $K$ to $\bar{K}$, and using the base change formula, we get a short exact sequence 
\begin{equation}
\label{eq:34}
0 \to \ENd{\cat{E}} \otimes \struct{X_K}(- \widetilde{Z}) \otimes_{K}
 \bar{K} \xrightarrow{\iota} \cat{A}(\cat{E}) \otimes_{K} \bar{K} \xrightarrow{\widehat{\sigma_1}} TX_{K}(- \log \widetilde{Z}) \otimes_{K} \bar{K} \to 0
\end{equation}
of vector bundles over $X_{\bar{K}}$.
Let $\widehat{\phi^B_{\cat{E}}}$ be the extension class
of \eqref{eq:34}, and  note that 
$$\coh{1}{X_{\bar{K}}}{\ENd{\cat{E}} \otimes \Omega^1_{X_K} \otimes \bar{K}} = 
\coh{1}{X_K}{\ENd{\cat{E}} \otimes \Omega^1_{X_K}} \otimes \bar{K},$$
and 
\begin{equation}
\label{eq:35}
\widehat{\phi^B_{\cat{E}}} = \phi^B_{\cat{E}} \otimes 1.
\end{equation}
Next,
the short exact sequence 
\eqref{eq:34} coincides with the short exact sequence 
\eqref{eq:14} associated with the vector bundle
$\cat{E}_{\bar{K}}$ over $X_{\bar{K}}$ with endomorphism
$\widehat{B(\hat{y})}$ on $\cat{E}_{\bar{K}}(\hat{y})$,
for every $\hat{y} \in \widehat{S}$, that is,
\begin{equation*}
0 \to \ENd{\cat{E}_{\bar{K}}} \otimes \struct{X_{\bar{K}}}(- \widehat{Z}) \xrightarrow{\iota} \cat{A}(\cat{E}_{\bar{K}}) \xrightarrow{\widehat{\sigma_1}} TX_{\bar{K}}(- \log \widehat{Z}) \to 0.
\end{equation*}
Therefore, the extension class $\phi^{\widehat{B(\hat{y})}}_{\cat{E}_{\bar{K}}}$ coincides with $\widehat{\phi^B_{\cat{E}}}$.

Since $\cat{E}_{\bar{K}}$
 admits a logarithmic connection singular over 
 $\widehat{S}$ with residue $\widehat{B(\hat{y})}$ for 
 every $\hat{y} \in \widehat{S}$, we have 
 $$0 = \phi^{\widehat{B(\hat{y})}}_{\cat{E}_{\bar{K}}} =\widehat{\phi^B_{\cat{E}}}.$$
Thus, from \eqref{eq:35}, $$\phi^B_{\cat{E}} \otimes 1 = 0,$$
and hence $\phi^B_{\cat{E}} = 0$. In view of Lemma 
\ref{lemm:4}, and Remark \ref{rem:2}, proof of $\eqref{a}$ is complete.

To prove the second part \eqref{b}, let $E$ be an
indecomposable vector bundle over $X$ which satisfies \eqref{eq:24}. Using the same technique as above, the 
vector bundle $\cat{E}$ over $X_K$ in \eqref{eq:29.5},
satisfies \eqref{eq:24}. Now, repeating the above 
argument, we conclude that $\cat{E}$ admits a logarithmic connection singular over $\widetilde{S}$
with residue $B(y)$, for every $y \in \widetilde{S}$.
Again, from Lemma \ref{lemm:4}, and Remark \ref{rem:2},
proof of the second part \eqref{b} of the theorem is 
complete. 
\end{proof}

\section{Conclusions}
The above theorems have been proved over any  perfect field and they will be very useful when we study the algebro-geometric invariants
for the moduli space of logarithmic connections 
over a perfect field, with fixed rigid residues.

In case residues are not rigid, 
finding a suitable criterion for the existence of a logarithmic connection is still an open problem.
Our guess is the same criterion should work but we do not 
know how to prove it.

\end{document}